\documentclass[12pt]{amsart}

\usepackage[left=1in,right=1in,top=1in,bottom=1in]{geometry} 

\usepackage{amsmath}

\usepackage{url}

\usepackage{amssymb,amsthm,amsmath}
\usepackage{subcaption}
\usepackage{ifthen}

\usepackage{graphicx}

\usepackage{enumerate}
\usepackage{tikz}
\usetikzlibrary{positioning}
\usetikzlibrary{arrows.meta, positioning, decorations.markings}

\DeclareGraphicsExtensions{.eps}

\newtheorem{theorem}{Theorem}[section]

\newtheorem{lemma}[theorem]{Lemma}

\newtheorem{corollary}[theorem]{Corollary}

\theoremstyle{definition}

\newtheorem{definition}[theorem]{Definition}

\newtheorem{example}[theorem]{Example}

\newtheorem{remark}[theorem]{Remark}

\numberwithin{equation}{section}

\begin{document}
	
\title[Regularity of the Semigroup of Regular Probability Measures]{Regularity of the Semigroup of Regular Probability Measures on Locally Compact Hausdorff Topological Groups and Applications}
\author{M. N. N. Namboodiri}
\address{Formerly of Department of Mathematics, Cochin University of Science \& Technology, Kochi, Kerala, India--682022\\ \&\\ IMRT Thiruvananthapuram.}
\email{mnnadri@gmail.com}

\subjclass[2000]{Primary 60B15}

\date{ABC XX, XXXX, and, in revised form, XYZ XX, XXXX.}

\keywords{Measures, Semigroup, Convolution}

	\begin{abstract}
		Let $\mathcal{G}$ be a locally compact Hausdorff group in which every element is of finite order, and let $P(\mathcal{G})$ denote the class of all regular probability measures on $\mathcal{G}$. In this note, it is observed that a characterization of algebraically regular elements in certain subsemigroups of $P(\mathcal{G})$ (Theorem~4.1 \cite{MNN}) for compact $\mathcal{G}$ remains true for locally compact $\mathcal{G}$. The results presented rely very much on Wendel's support theorem and the theory of doubly stochastic matrices (\cite{PLC} \cite{WEN}). A more general situation in which the assumption that every element is of finite order is not satisfied is also considered and illustrated with an example. This short article is a followup of \cite{MNN}.In addition,the random walk on an infinite discrete group associated with the generalised inverse of a probability measure is also studied.The last section is devoted to Fourier transform techniques  to study the same problem using Bochner-type theorem in harmonic analysis.
	\end{abstract}
	\maketitle
\section{Introduction}
Let $\mathcal{G}$ be a locally compact Hausdorff group, and let $P(\mathcal{G})$ denote the set of all regular probability measures on $\mathcal{G}$. It is well known that $P(\mathcal{G})$ is a semigroup under convolution, and it is abelian if and only if $\mathcal{G}$ is abelian. Furthermore, $P(\mathcal{G})$ is a compact, convex set under the weak topology of measures. This topological semigroup has been intensively studied by many researchers over the years. For example, refer to the classic monographs by Parthasarathy \cite{KRP}, Grenander \cite{UFG1}, and Mukherjea and Tserpes \cite{MNT}.

Of particular relevance to this paper, Wendel \cite{WEN} established many significant results regarding the algebraic, topological, and geometrical structures of $P(\mathcal{G})$ when $\mathcal{G}$ is compact. In particular, he showed that it is a closed, convex semigroup, which is not a group except when $\mathcal{G} = \{e\}$, the trivial group, by demonstrating that the only invertible elements are the point-mass measures. The problem we consider is the algebraic regularity of $P(\mathcal{G})$. A semigroup is algebraically regular if every element has a generalized inverse.
The following generalization of Theorem~4.1 \cite{MNN} can be proved for locally compact groups using the same techniques. We therefore omit the proof.
       
       		\begin{theorem}
       	Let $\mathcal{G}$ be a locally compact Hausdorff topological group such that every element in $\mathcal{G}$ is of finite order. Let
       	$M(\mathcal{G}) = \text{Conv}\{\delta_{g}: g \in \mathcal{G}\}$,
       	where \textup{`Conv'} denotes the convex hull. For an arbitrary finite set $\{g_{1}, g_{2}, \ldots, g_{n}\} \subset G$,
       	\begin{equation}
       		\mu = \frac{\delta_{e} + \delta_{g_{1}} + \delta_{g_{2}} + \cdots + \delta_{g_{n}}}{n + 1}
       	\end{equation}
       	is regular in $M(\mathcal{G})$. Moreover, the elements of the following form, namely,
       	\begin{equation}
       		\{\delta_{g} \star \mu \star \delta_{h}: g, h \in G \quad \& \quad \mu = \frac{\delta_{e} + \delta_{g_{1}} + \delta_{g_{2}} + \cdots + \delta_{g_{n}}}{n + 1}\},
       	\end{equation}
       	are also regular in $M(\mathcal{G})$.
       \end{theorem}
   From now on, $\mathcal{G}$ will denote \textbf{a locally compact Hausdorff topological group} such that every element in it is of \textbf{finite order}. Special cases shall be mentioned in addition.
\section{Preliminaries}\label{sec2}
Let $\mathcal{B}$ denote the $\sigma$-algebra of all Borel sets in $\mathcal{G}$. A probability measure $\mu$ is a nonnegative countably additive function on $\mathcal{B}$ such that the total mass $\mu(\mathcal{G}) = 1$. A point mass measure or Dirac delta measure is a measure $\mu$ for which there is an element $x \in \mathcal{G}$ such that $\mu(A) = 1$ if $x \in A$ and zero otherwise, $A \in \mathcal{B}$. Such a measure is usually denoted by $\delta_{x}$. One of the interesting results of Wendel is that the only invertible elements in $P(\mathcal{G})$ are Dirac delta measures.

The convolution product $\star$ in $P(\mathcal{G})$ is defined as follows:

\begin{definition}(\textbf{Convolution})
	Let $\mu, \nu \in P(\mathcal{G})$. Then $\mu \star \nu$ is the probability measure defined as $\mu \star \nu (A) = \int \mu(Ax^{-1}) \,d\nu(x)$
	for every $A \in \mathcal{B}$.
\end{definition}

\begin{definition}(\textbf{Generalized Inverse})
	Let $\mathcal{S}$ be a semigroup, and let $s \in \mathcal{S}$. An element $s^{\dagger} \in \mathcal{S}$ is called a generalized inverse of $s$ if $ss^{\dagger}s = s$.
\end{definition}

	\begin{remark}
		In a general semigroup $\Omega$, if $\omega\in \Omega$ has a generalized inverse, then it has a stronger type of inverse called \textbf{semi-inverse} \cite{WJR}: to be explicit, if $\omega g \omega=\omega$ for some $g\in \Omega$, then 
		\begin{equation}
			\omega \omega^{\dagger} \omega=\omega,
		\end{equation}
		\begin{equation}
			\omega^{\dagger}\omega \omega^{\dagger}=\omega^{\dagger},
		\end{equation}
		where $\omega^{\dagger}=g \omega g$. So, to characterize generalized invertibility, it suffices to characterize semi-inverses.
	\end{remark}
	In an abstract semigroup with an involution say $\kappa$ , the Moore-Penrose if it exists then it is unique. Let us recall the defenition of Moore-Penrose inverse in this abstract setting.
	\begin{definition} 
		\begin{enumerate}
			Let $\mathcal{H}$ be a semigroup and $x\in\mathcal{H}$. An element $y\in\mathcal{H}$ is called Moore-Penrose of $x$ if
			\item: $xyx=x$\\
			\item: $yxy= y$ \\
			\item: $(yx)^{\kappa}= yx$\\
			and
			\item:$(xy)^{\kappa} = xy$.
		\end{enumerate}
	\end{definition} 
	It is well known that the Moore-Penrose inverse exists, then it is unique.
	The following lemma is proved in (Lemma 3.5) \cite{MNN}.
	\begin{lemma}\cite{MNN}\label{lem9}
		Let $\mu$ be a regular element in $P(\mathcal{G})$ with Moore--Penrose inverse $\mu^{\dagger}$. Suppose $e\in S(\mu)$ and every element of $S(\mu)$ is of finite order. Then $S(\mu)$ is a subgroup, and $S(\mu)= S(\mu^{\dagger})$. 
	\end{lemma}
	\section{Geometric Characterization of Generalized Invertible Elements in $P(\mathcal{G})$}
	
In the subsequent portion of this article, it is assumed that the supports of the measures, and consequently their products, are finite. Therefore, the closures taken will remain unchanged. This simplifies the arguments in the subsequent proofs. The following theorem is a locally compact version of Theorem~4.1 in \cite{MNN}.
	\begin{theorem}
		Let $\mathcal{G}$ be a locally compact Hausdorff topological group such that every element in $\mathcal{G}$ is of finite order. Let
		\begin{equation}
			M(G) = \text{Conv}\{\delta_{g}: g \in G\},
		\end{equation}
		where \textup{`Conv'} denotes the convex hull. For an arbitrary finite set $\{g_{1}, g_{2}, \ldots, g_{n}\} \subset \mathcal{G}$,
		\begin{equation}
			\mu = \frac{\delta_{e} + \delta_{g_{1}} + \delta_{g_{2}} + \cdots + \delta_{g_{n}}}{n + 1}
		\end{equation}
		is regular in $M(\mathcal{G})$. Moreover, the elements of the following form, namely,
		\begin{equation}
			\{\delta_{g} \star \mu \star \delta_{h}: g, h \in \mathcal{G} \quad \& \quad \mu = \frac{\delta_{e} + \delta_{g_{1}} + \delta_{g_{2}} + \cdots + \delta_{g_{n}}}{n + 1}\},
		\end{equation}
	are also regular in $M(\mathcal{G})$.
	\end{theorem}
\begin{proof}
	The proof is identical to that of Theorem~4.1 in \cite{MNN} and is therefore omitted.
\end{proof}
\begin{remark}
	We may apply the generalized doubly stochastic invertibility due to Plemmons and Cline, Theorem~2 \cite{PLC}, to solve for (probabilistic vector) stochastic solutions of the  matrix equations as done in \cite{MNN}  ; this will provide the generalized invertible probability measures we are looking for.
\end{remark}
\section{Jonsson-Type Groups}
We state the following theorem due to Pym, Theorem~4.1 in \cite{JSP}. In this section, the locally compact groups $\mathcal{G}$ \textbf{need not} satisfy the requirement that every element is of finite order.
\begin{theorem}
	A positive idempotent measure on a locally compact group is the Haar measure of a compact subgroup, and conversely.
\end{theorem}
\begin{remark}
	It is well known that every countable, locally compact Hausdorff group is necessarily a discrete space, and therefore compact subsets of such topological groups are always finite. This can be used to prove the following theorems.
\end{remark}
Next, we consider a more general case in which the probability measures have countable support.
\begin{theorem}
	Let $P_{\aleph_{0}}(\mathcal{G})$ denote the class of all regular probability measures with countable support and containing the identity $e$. Let $\mu \in P_{\aleph_{0}}(\mathcal{G})$, and if $\mu^{\dagger}$ is the Moore--Penrose inverse of $\mu$, then $S(\mu^{\dagger})$ is a finite subgroup of $G$.
\end{theorem}
  
 	 \begin{proof}
	Since $\mu\star \mu^{\dagger}$ as well as $\mu^{\dagger}\star \mu$ are idempotents, by Theorem~4.1 we have that $S(\mu\star \mu^{\dagger})$ as well as $S(\mu^{\dagger}\star \mu)$ are finite subgroups. Moreover, we have 
	$$
	S(\mu)^{\dagger}S(\mu)S(\mu^{\dagger}) \subset S(\mu^{\dagger}),
	$$
	which implies that 
	$$
	S(\mu)^{\dagger}S(\mu^{\dagger}) \subset S(\mu^{\dagger})
	$$
	since $e\in S(\mu)$. Thus $S(\mu^{\dagger})$ is a semigroup. By similar reasoning, $S(\mu^{\dagger}) \subset S(\mu\star \mu^{\dagger})$. Therefore, $S(\mu^{\dagger})$ (which is a subset of $S(\mu)$) is a finite semigroup, and hence it is a finite group. Similarly, we can show that $S(\mu)$ is a finite subgroup that equals $S(\mu^{\dagger})$.
\end{proof}

\begin{corollary}
	Let $\{ g_{1}, g_{2},\ldots\}$ be an infinite subgroup of $\mathcal{G}$. If $\mu=\sum_{n=1}^{\infty} \alpha_{n}\delta_{g_{n}}$, where $\alpha_{j}> 0$ for infinitely many $j$, then $\mu$ is not generalized invertible.
\end{corollary}
\begin{proof}
	Follows immediately from Theorem~4.3.
\end{proof}
\begin{remark}
	Theorem~4.3 and its corollary imply that the generalized invertible elements in $P_{\aleph_{0}}(\mathcal{G})$ can be obtained using the Perron--Frobenius theory, as in the earlier case. 
	As mentioned above, one needs to investigate the generalized invertibility of infinite stochastic matrices. It is not clear whether results of the type due to J.~R.~Wall or Plemmons and Cline are known for infinite stochastic matrices. We therefore analyse the problem by examining special cases of infinite stochastic matrices in order to shed light on this seemingly complicated situation.
\end{remark}
	\begin{example}
		There are important examples of such groups; for instance, the \textbf{Grigorchuk group}. This group has growth that is faster than polynomial but slower than exponential. Grigorchuk constructed this group in a 1980 paper and proved that it has intermediate growth in 1984 \cite{GRK}. In fact, the above-cited group is a locally compact topological metric group under the word metric, and every element has finite order.
	\end{example}
\begin{example}
The Jonsson group of cardinality $\aleph_{1}$ equipped with the discrete topology \cite{SHS} is an example illustrating Corollary~4.4 above. Since the supports of idempotent probability measures are compact, all regular elements in this special case have finite support. However, in this group \textbf{not} all elements are of finite order.
	\end{example}
\section {Random Walk on Finitely generated Infinite Discrete Groups}
Random walk on discrete groups has been of considerable interest,theoritical as well as application point of view see  \cite{SPL} , \cite{SPL1} and \cite{LSC} for instance.In this section we study the random walk associted with the Moore-Penrose inverse $\mu^{\dagger}$ of a probability measure $\mu$ on an an infinite discrete group, provided $\mu^{\dagger}$ exists. For a toological group$G$ and $\mu\in P(G)$, the reflected measure $\tilde{\mu}$ is the probability measure defined as
$$
\tilde{\mu}(E)=\mu(E^{-1})
$$
for every Borel set $E\subset G$. This measure(exists for any measure) is also seem to be called the generalised inverse measure in a different sense. The random walk corresponding to $\tilde{\mu}$ the time reversed process. However, our generalised inverse need not exist for all probability measure.

We restrict our attention to disrete group $\mathcal{G}$.First we recall a few results related to the Moore-Penrose inverse in an abstract semigroup with an involution.
Wendel introduced the following involution $\star$ in $P(\mathcal{G})$:
Let $\nu \in P(\mathcal{G})$. Then the involution $\nu^{\star}$ of $\nu$ is defined as
$$
\nu^{\star}(E)= \nu(E^{-1}), 
$$
for every Borel set $E$ in $G$. The following Lemma can be found in Lemma2\cite{WEN}.
\begin {lemma}
$\nu \in P(\mathcal{G})$ is an idempotent  then 
$\nu^{\star}= \nu$.
\end{lemma}
 Moreover the support $S(\mu)$=$S(\mu^{\dagger} )$=$\{g_{0}=g_{1},...,g_{n-1}\}$,  a finite subgroup of $\mathcal{G}$, where $g_{0}= e$ is the identity element in $\mathcal{G}$. Hence it is enough to consider probability measures on finite groups for the present discussion.

Let us recall the definition of a \textbf{random walk} on $\mathcal{G}$.
A \textbf{random walk} on $\mathcal{G}$  is an infinite sequence $\{X_{0}, X_{1}, · · ·\}$ of $\mathcal{G}$-valued random variables, all defined on the same probability space $(\Omega, \mathcal{F} , P^{x}  )$ of the form
$$
X_{n}= X_{0}\zeta_{0}\zeta_{1}...\zeta_{n}
$$
where $\{\zeta_{0},\zeta_{1},...,\zeta_{n}\}$ are identically distrbuted $\mathcal{G}$ valued random variables with respect to $\mu$. A special case of interest shall be the following;
 
$$
\zeta_{k}(g)= g^{k}.h,  \ g\in \mathcal{G}
$$
where $h$ is a randmly chosen element in $\mathcal{G}$ for $\{k=1,2,...,n\}$.
The Cayley graph of the reflection measure 
on a finite group 
is a weighted directed graph defined on the Cayley Graph with vertices $\{g_{0}, g_{1},..., g_{n}\}$ and the vertices $(g_{k},g_{k+1})$ are connected if $g _{k+1}= g_{k}h$,where $h$ is a randomly chosen element of $\mathcal{G}$. If $\mu$ is a probability measure on $\mathcal{G}$,then then the right random walk corresponding to the relexion $\bar{\mu}$ of $\mu$ is called the 'the reverse Markov' chain on $\mathcal{G}$ on $\mathcal{G}$. To be more precise, we compute the time homogeneous Markov chain  associated with the probability measure $\mu^{\dagger}$ by determining the corresponding transition probability.

Next we compute the random walk ov the Moore-Penrose inverse $\mu^{\dagger}$ of $\mu$, whenever it exists. Before proceeding to the statement of the result presented below, we provide a few essential concernng Theorem 4.1 \cite{MNN}. Let $\mu= \Sigma_{k=0}^{n}\alpha_{k} \delta_{g_{k}}: \alpha_{k}> 0$ and $\mu^{\dagger}= \Sigma_{k=0}^{n}\beta_{k}\delta_{g_{k}}$ be the Moore-Penrose inverse which assumed toexist.Next we recall equations  4-18 to 4.20,\cite{MNN}. The matrix $A$ in equation 4.18 quoted above is as follows:
	\begin{equation}
	\textbf{A} =
	\begin{bmatrix}
		\alpha_{S_{(0)}(0)}\,\alpha_{S_{(0)}(1)},\quad \ldots \alpha_{S_{(0)}(n)} \\
		\alpha_{S_{(1)}(0)}\,\alpha_{S_{(1)}(1)},\quad \ldots \alpha_{S_{(1)}(n)}\\
		.\\ \quad .\\
		
		. \quad .,           .      \quad .                              \quad       .
		
		\\
		\alpha_{S_{(n)}(0)},\alpha_{S_{(n)}(1
			)},\quad \ldots \alpha_{S_{()}(n)}
	\end{bmatrix},
\end{equation} 
where $S_{j}$ is the permutation on $\{0,1,2,...n.\}$ given by $g^{-1}_{k}g_{j}\rightarrow g_{S_{j}(k)}$, $j,k\in\{0,1,2,...,n\}$. It can be easily seen that the matrix $A$ is doubly stochastic and therefore $A^{2}$ is doubly stochastic. Hence by Plemons and Clines,if $A$ has doubly stochastic Moore-Penrose inverse ,then by  Theorem 2 \cite{PLC}, the Moore-Penrose inverse of $A$ is its transpose $A^{t}$.Hence ${A^{2}}^{\dagger}= {A^{t}}^{2}$. In case $A$ is Hermitian, then ${A^{\dagger}}^{2}= AA^{t}$ is an orthogonal projection hence we may apply Doob's theorem \cite{JLD} or Lemma1 \cite{PLC}. However the following theorem shows that the hermicity requirement is satisfied only when the underlying group is abelian.

\begin{theorem}
	Let $\mathcal{G}$ be a finite group and $A$ be the matrix given by equation (5.1). Then $A^{2}$ is hermitian if and only if $\mathcal{G}$ is abelian.
	\end{theorem}
\begin{proof}
	It is clear that if $\mathcal{G}$ is abelian then $A$ and hence $A^{2}$ is abelian. Conversely assume that $A^{2}$ is hermitian. Let $A= A_{1}+ i A_{2}$ where $A_{1} \& A_{2}$ are the real and imaginary parts of $A$. By assumption, $A^{2}= A_{1}^{2}-A_{2}^{2}+ i(A_{1}A_{2}+ A_{2}A_{1})$ is hermitian which implies that $(A_{1}A_{2}+ A_{2}A_{1})=0$. Hence $A_{1}A_{2}=- A_{2}A_{1}$. But  $(A_{1} \& A_{2}$ are doubly stochastic and hence $A_{1}A_{2} \& A_{2}A_{1}$ are also doubly stochastic. Therfore we will have $A_{1}A_{2}= A_{2}A_{1}=0$. This in turn implies that either $A_{1}=0$ or $A_{2}=0$ being doubly stochastic and $\alpha_{k} > 0$ for all $k= 1,2,...n$.
	 Therefore $A_{2}$ has to be $0$ and hence $\mathcal{G}$ is abelian.
	\end{proof}

\begin{remark}
	The above theorem shows that to deal with non abelian groups , one need more sophisticated 
	results established for nonhermitian doubly stochastic matrices by Plemons and Cline \cite{PLC}. 
	\end{remark}

\begin{theorem}
	Let $\{\mu\, \mu^{\dagger}, A, A^{\dagger} \}$ be as given above. Then the transition probability matrix of the homogeneous Markov-chain corresponding to $\mu^{\dagger}$ is determined by the matrix $A^{t}$.
	\end{theorem}
\begin{proof}
	Recalling the equation determining $\mu^{\dagger}$, we have
		\begin{equation}
		\begin{bmatrix}
			\alpha_{S_{(0)}(0)}\,\alpha_{S_{(0)}(1)},\quad \ldots \alpha_{S_{(0)}(n)} \\
			\alpha_{S_{(1)}(0)}\,\alpha_{S_{(1)}(1)},\quad \ldots \alpha_{S_{(1)}(n)}\\
			.\\ \quad .\\
			
			. \quad .,           .      \quad .                              \quad       .
			
			\\
			\alpha_{S_{(n)}(0)},\alpha_{S_{(n)}(1)},\quad \ldots \alpha_{S_{(n)}(n)}
		\end{bmatrix}^{2}
		\begin{bmatrix}
			\beta_{0}\\
			
			\beta_{1}\\
			
			\vdots\\
			\beta_{k}\\
			\vdots\\
			
			\beta_{n}    
		\end{bmatrix}
		=
		\begin{bmatrix}
			\alpha_{0}  \\
			\alpha_{1}\\
			
			\vdots\\
			\alpha_{k}\\
			\vdots\\
			\alpha_{n}
		\end{bmatrix}. 
	\end{equation}
Therefore we have,
\begin{equation}
	\begin{bmatrix}
	\beta_{0}\\
	
	\beta_{1}\\
	
	\vdots\\
	\beta_{k}\\
	\vdots\\
	
	\beta_{n}    
\end{bmatrix}
=
transpose{\begin{bmatrix}
	\alpha_{S_{(0)}(0)}\,\alpha_{S_{(0)}(1)},\quad \ldots \alpha_{S_{(0)}(n)} \\
	\alpha_{S_{(1)}(0)}\,\alpha_{S_{(1)}(1)},\quad \ldots \alpha_{S_{(1)}(n)}\\
	.\\ \quad .\\
	
	. \quad .,           .      \quad .                              \quad       .
	
	\\
	\alpha_{S_{(n)}(0)},\alpha_{S_{(n)}(1)},\quad \ldots \alpha_{S_{(n)}(n)}
\end{bmatrix}}^{2}
\begin{bmatrix}
	\alpha_{0}  \\
	\alpha_{1}\\
	
	\vdots\\
	\alpha_{k}\\
	\vdots\\
	\alpha_{n}
\end{bmatrix}. 
\end{equation}.
Therefore we have $\mu^{\dagger}= \Sigma_{k=0}^{n} \beta_{k}\delta_{g_{k}}$.Hence the corresponding  random walk is determined by the transition probability matrix is 
\begin{equation}
p_{r}(i,j)= \mu^{\dagger} g_{i}^{-1} g_{j}). 
\end{equation}
	\end{proof}
 Random walks corresponding to some specific examples provide a rich supply of concrete cases. examples.
\begin{example}
	\begin{enumerate}
		\item: Abelian groups and cyclic groups,\\
		\item: Gregorschuk group,\\
		\item: The rotation group $SO(3)$,\\ 
		\item: Dihedral group, and \\
		\item : Permutation group.
		\end{enumerate}
	Next we compute the random walk corresponding to probability measures on groups of order four , which are abelian by Lagrange's. As for non-abelian case, we compute random walk on  $SO(3)$. Recall the matrix $A$ in Example 5.5 in \cite{MNN}
	\end{example}
	
				\begin{example}
				\begin{enumerate}
					\item:
					The case $n=1$ has already been done.Now consider the case $n=2$ so that $G=\{e,g_{1},g_{2} : g^{2}_{k}=e\}$.	In this case we will have $g_{1}g_{2}=e,or g_{1} or g_{2}$,which implies that $G$ is a two element one.
					\item: The Cyclic Group of Order Four namely $\mathcal{G}=\{e,g,g^2,g^3\}$
					\ The corresponding multiplication table is the following and Figure 1:

					\begin{center}
						\begin{tabular}{ c | c | c | c | c |}
							$*$ & $e$ & $g$ & $g^{2}$ & $g^{3}$ \\
							\hline
							$e$ & $e$ & $g$ & $g^2$ &$g^3$\\
							\hline
							$g$ & $g$ & $g^{2}$ & $g^3$ &$e$\\
							
							\hline
							$g^{2}$ & $g^{2}$ & $g^{3}$ & $e$ &$g$\\
							\hline
							$g^{3}$ & $g^{3}$ & $e$ & $g$&$g^{2}$ \\
							\hline
						\end{tabular}
					\end{center}.
					\vspace{0.8cm}

					\begin{center}
						\begin{tikzpicture}[
							>=stealth,
							vertex/.style={circle, draw, minimum size=8mm},
							lab/.style={font=\scriptsize, fill=white, inner sep=1pt}
							]
							scale=6.5,
							vertex/.style={circle, fill=black!10, draw=black, thick, minimum size=7mm, inner sep=0pt},
							rededge/.style={red, -{Stealth}, thick},
							blueedge/.style={blue, -{Stealth}, thick}
							]
							\node[vertex] (e)  at (-2, 2)  {$e)$};
							\node[vertex] (g)  at ( 2, 2)  {$g$};
							\node[vertex] (g2) at ( 2,-2)  {$g^2$};
							\node[vertex] (g3) at (-2,-2)  {$g^3$};
							\draw[->] (e) to[ bend left=10]
							node[lab, above] {$\mu^\dagger(g^3)$} (g);
							
							\draw[->] (g) to[bend left=10]
							node[lab, below] {$\mu^\dagger(g)$} (e);
							
							\draw[->] (g) to[bend left=10]
							node[lab, right] {$\mu^\dagger(g^3)$} (g2);
							
							\draw[->] (g2) to[bend left=10]
							node[lab, left] {$\mu^\dagger(g)$} (g);
							
							\draw[->] (g2) to[bend left=10]
							node[lab, below] {$\mu^\dagger(g^3)$} (g3);
							
							\draw[->] (g3) to[bend left=10]
							node[lab, above] {$\mu^\dagger(g)$} (g2);
							
							\draw[->] (g3) to[bend left=10]
							node[lab, left] {$\mu^\dagger(g^3)$} (e);
							
							\draw[->] (e) to[bend left=10]
							node[lab, right] {$\mu^\dagger(g)$} (g3);
							
							\draw[<->] (e) --
							node[lab, pos=.20, below right]
							{$\mu^\dagger(g^2)$} (g2);
							
							\draw[<->] (g) --
							node[lab, pos=.20, below left]
							{$\mu^\dagger(g^2)$} (g3);
							
						\end{tikzpicture}
						
						\vspace{0.3cm}
						
						{\scriptsize Figure 1: Cyclic group $C_4$}
					\end{center}
					
					\vspace{0.8cm}
					
					\item: The Klein Four Group $V_{4}$

					Now we consider the case $n=3$.The following  (when $\mathcal{G}$ is not cyclic)  well known relations, namely
					\begin{equation}
						g_{1}g_{2}=g_{3},g_{2}g_{3}=g_{1} \& g_{1}g_{3}=g_{2}.
					\end{equation}
					implies that $g_{1}^{2}=g_{2}^{2}=g_{3}^{2}= e$	
					The above relations determines the permutations $S_{0},S_{1},S_{2} \&S_{3}$ on $\{0,1,2,3.\}$.Simple computations reveals that the matrix $\textbf{A}$ is as follows. Next we
					recall the multiplication table for $V_{4}$.
					\begin{center}
						\begin{tabular}{ c| c | c | c | c |c|c|c|c|c|}
							$*$ & $g_0$ & $g_1$ & $g_2$ & $g_3$ \\
							\hline
							$g_0$ & $g_0$ & $g_1$ & $g_2$ & $g_3$ \\
							\hline
							$g_1$ & $g_1$ & $g_0$ & $g_3$ &$g_2$\\
							\hline
							$g_2$ & $g_2$ & $g_3$ & $g_0$ &$g_1$\\
							\hline
							$g_3$ & $g_3$ & $g_2$ & $g_1$&$g_0$ \\
							\hline
						\end{tabular}
					\end{center}.

					The corresponding transition probability matrix is as follows and Figure 2;
					\begin{equation}
						\textbf{A}= \begin{bmatrix}
							\alpha_{0}\quad \alpha_{1}\quad \alpha_{2} \quad \alpha_{3}\\
							\alpha_{1}\quad \alpha_{0}\quad \alpha_{3} \quad \alpha_{2}\\
							\alpha_{2}\quad \alpha_{3}\quad \alpha_{0} \quad \alpha_{1}\\
							\alpha_{3}\quad \alpha_{2}\quad \alpha_{1} \quad \alpha_{0}
						\end{bmatrix},
					\end{equation}
					which is a Hermetian doubly stochastic matrix.	
					
					In addition, we have the subsequent equation 
					\begin{equation}
						{\begin{bmatrix}
								\alpha_{0} \quad \alpha_{1}\quad \alpha_{2}\quad\alpha_{3}\\   	                                                                                             
								\alpha_{1}\quad \alpha_{0}\quad \alpha_{3} \quad \alpha_{2}\\  
								\alpha_{2}\quad \alpha_{3}\quad \alpha_{0} \quad \alpha_{1}\\
								\alpha_{3}\quad \alpha_{2}\quad \alpha_{1} \quad \alpha_{0}
						\end{bmatrix}}^{2}
						\begin{bmatrix}
							\beta_{0}\\
							\beta_{1}\\
							\beta_{2}\\
							\beta_{3}
						\end{bmatrix}
						=
						\begin{bmatrix}
							\alpha_{0}\\
							\alpha_{1}\\
							\alpha_{2}\\
							\alpha_{3}
						\end{bmatrix}.	
					\end{equation}
					where  $\alpha_{j}: j=0,1,2,3$ are distinct positive numbers such that $\Sigma_{k=0}^{3}\alpha_{k}=1$. But $A^{2}=AA^{\dagger}$. Being hermitian and doubly stochastic, $A^{\dagger}$ is the Moore-Penrose inverse of $A$. Hence $A^{2}$is is an orthogonal projection and therefore its Moore-Penrose inverse is itself. Thus the solution we seek for is 
					\begin{equation}
						{\begin{bmatrix}
								\alpha_{0} \quad \alpha_{1}\quad \alpha_{2}\quad\alpha_{3}\\   	                                                                                             
								\alpha_{1}\quad \alpha_{0}\quad \alpha_{3} \quad \alpha_{2}\\  
								\alpha_{2}\quad \alpha_{3}\quad \alpha_{0} \quad \alpha_{1}\\
								\alpha_{3}\quad \alpha_{2}\quad \alpha_{1} \quad \alpha_{0}
						\end{bmatrix}}^{2}
						\begin{bmatrix}
							\alpha_{0}\\
							\alpha_{1}\\
							\alpha_{2}\\
							\alpha_{3}
						\end{bmatrix}
						=
						\begin{bmatrix}
							\beta_{0}\\
							\beta_{1}\\
							\alpha_{2}\\
							\beta_{3}
						\end{bmatrix}.
					\end{equation}
					The above equation(5.7) gives $\mu^{\dagger}=\Sigma_{j=0}^{3}\beta_{j}\delta_{g_{j}}$. Now the random walk on the Cayley graph of the group $V_{4}$ determined by $\mu^{\dagger}$
					is given by the transition probability matrix $A^{2}$ which has an explicit form by Doob's theorem.
				\end{enumerate}
				
				The following diagram illustrates the random walk  indiced by $\mu^{\dagger}$ on the Cayley graph of $V_{4}$, where $\mu$ is a probability measure on $V_{4}$ such that $e=g_{0}\in S(\mu)$. Observe that any probability measure on $V_{4}$ is symmetric since $g^{2}=e$ for all $g\in V_{4}$.
				
				\vspace{0.8cm}
				
				\begin{center}
					\begin{tikzpicture}[
						vertex/.style={circle, draw, minimum size=8mm},
						every node/.style={font=\small},
						>=stealth
						]
						scale=6.5,
						vertex/.style={circle, fill=black!10, draw=red,blue, thick, minimum size=7mm, inner sep=0pt},
						rededge/.style={red, -{Stealth}, thick},
						blueedge/.style={blue, -{Stealth}, thick}
						]
						
						\node[vertex] (g0) at (0,2) {$g_0$};
						\node[vertex] (g1) at (4,2) {$g_1$};
						\node[vertex] (g2) at (0,0) {$g_2$};
						\node[vertex] (g3) at (4,0) {$g_3$};
						\draw[<->](g0) -- node[above]{$\mu^\dagger(g_1)$} (g1);
						\draw[<->] (g2) -- node[below] {$\mu^\dagger(g_1)$} (g3);
						\draw[<->] (g0) -- node[left] {$\mu^\dagger(g_2)$} (g2);
						\draw[<->] (g1) -- node[right] {$\mu^\dagger(g_2)$} (g3);
						
						\draw[<->] (g0) -- node[pos=.37,above right] {$\mu^\dagger(g_3)$} (g3);
						\draw[<->] (g2) -- node[pos=.37,below right] {$\mu^\dagger(g_3)$} (g1);
						
					\end{tikzpicture}
					
					\vspace{0.3cm}
					
					{\scriptsize Figure 2: Klein group}
				\end{center}

				\vspace{0.8cm}

			\end{example}
			
			\begin{example}{Permutation Group $S_{3}(O)$}
				We compute the doubly stochastic matrix in (5.1) for $n=5$ for which the multiplication table is essential. For $S_{3}(O)$ the following multiplication table is well known. 
				Here the group elements are denoted as $\{ g_0,g_1,g_2,g_3,g_4,g_5\}$ for convenience where
				$$
				\{g_{0}= i_d,g_1=(1 2),g_2=(1 3),g_3=(2 3),g_4=(1 2 3),g_4=(1 3 2)\}.
				$$

				Multiplication Table :-$\rightarrow$

				\begin{center}
					\begin{tabular}{ c| c | c | c | c |c|c|c|c|c|}
						$*$ & $g_0$ & $g_1$ & $g_2$ & $g_3$ &$g_4$ &$g_5$\\
						\hline
						$g_0$ & $g_0$ & $g_1$ & $g_2$ & $g_3$ &$g_4$ &$g_5$\\
						\hline
						$g_1$ & $g_1$ & $g_0$ & $g_4$ & $g_5$ & $g_2$ &$g_3$\\
						\hline
						$g_2$ & $g_2$ & $g_3$ & $g_0$ & $g_1$ &$g_5$ &$g_4$\\
						\hline
						$g_3$ & $g_3$ & $g_2$ & $g_5$ & $g_4$ &$g_0$ &$g_1$\\
						\hline
						$g_4$ & $g_4$ & $g_5$ & $g_1$ & $g_0$ &$g_3$ &$g_2$\\
						\hline
						$g_5$ & $g_5$ & $g_4$ & $g_3$ & $g_2$ &$g_1$ &$g_0$\\
						\hline
					\end{tabular}
				\end{center}
				A pair $(g_{i}, g_{j})$ is connected if $g_{j}=g_{i}h$ for some $h\in \mathcal{G}$. We can determine the connected pairs from the above multilication table and Figure 3

				The permutation $S(k)_{j}$ determines the product $g_{j}g_{k}=g_{S(k)_{j}}$. This again can be readoff from the above given multiplication table.
				
				\begin{center}
					\begin{tikzpicture}[
						scale=3.5,
						vertex/.style={circle, fill=black!10, draw=black, thick, minimum size=7mm, inner sep=0pt},
						rededge/.style={red, -{Stealth}, thick},
						blueedge/.style={blue, -{Stealth}, thick}
						]
						\node[vertex] (g0) at (90:1.5) {$g_0$};
						\node[vertex] (g1) at (30:1.5) {$g_1$};
						\node[vertex] (g2) at (330:1.5) {$g_2$};
						\node[vertex] (g3) at (270:1.5) {$g_3$};
						\node[vertex] (g4) at (210:1.5) {$g_4$};
						\node[vertex] (g5) at (150:1.5) {$g_5$};
						
						
			\draw[rededge, bend left=15] (g4) to (g3);
			\draw[blueedge, bend left=15] (g3) to (g4);
			
			\draw[rededge, bend left=15] (g4) to (g5);
			\draw[blueedge, bend left=15] (g5) to (g4);
			
			\draw[blueedge, bend left=15]  (g3) to (g2);
			\draw[rededge, bend left=15] (g2) to (g3);
			\draw[blueedge, bend left=15]  (g2) to (g1);
			\draw[rededge, bend left=15] (g1) to (g2);
			
			\draw[rededge, bend left=15]  (g1) to (g0);
			\draw[blueedge, bend left=15] (g0) to (g1);
			
			\draw[rededge, bend left=15]  (g0) to (g3);
			\draw[rededge, bend left=15] (g3) to (g0);
			
			\draw[rededge, bend left=15]  (g0) to (g2);
			\draw[blueedge, bend left=15] (g2) to (g0);
			
				\draw[blueedge, bend left=15]  (g5) to (g1);
			\draw[blueedge, bend left=15] (g1) to (g5);
				\draw[rededge, bend left=15]  (g5) to (g0);
			\draw[blueedge, bend left=15] (g0) to (g5);
				\draw[rededge, bend left=15]  (g4) to (g1);
			\draw[blueedge, bend left=15] (g1) to (g4);
			
		\end{tikzpicture}
		\vspace{0.3cm}
	
	{\scriptsize Figure 3: Permutation Group $S_3$}
	\end{center}
\begin{remark}
	In Figure 3 above the edges are connected according the rule 5.4 which is a right random walk induced by $\mu^{\dagger}$ and we opt not to mark on the edges which can to lead overcrowding.
	\end{remark}
	
	\end{example}

\section{Non commutative Fourier Transform}
In this section the wellknown noncommutative Fourier transform and the corresponding non commutative Bochner's theorem is applied to identify regular elements in $P(G)$ when $G$ is a compact Lie group.
Let $G$ be a locally compact Hausdorff topological group $G$.The unitary dual space 	$\widehat{G}$ of $G$ is defined as follows;
\begin{equation}
	\widehat{G}=\{(\pi,H_{\pi})\},	
\end{equation}

where $\pi:G\rightarrow B(H_{\pi}),{\pi}$  is unitary, irreducible representation of $G$ on a complex separable Hilbert space $H_{\pi}$ with the identification by unitary equivalence of representations.	It is also well-known that when $G$ is compact, each $H_{\pi}$	is finite-dimensional. That means that the dimension  $d_{\pi}$ of $H_{\pi}$ is finite and $d_{\pi}=1$ if $G$ is abelian.
The \textbf{Fourier transform} of  $\mu \in P(G)$ is defined as the function  $\hat{\mu}:\widehat{G}\rightarrow B(H_{\pi})$ defined by

\begin{equation}
	\hat{\mu}(\pi)\psi=\int_{G} \pi(g^{-1})(\psi) \mu(dg),
\end{equation}	
$\pi \in \widehat{G} \& \psi \in H_{\pi}$.
For a compact Hausdorff  group $G$ ,let $\mathcal{M}=\cup_{d_{\pi}}	M_{d\pi}(\mathbb{C}).$
A map $\Phi:\widehat{G} \rightarrow \mathcal{M}(\widehat{G})$ is called \textbf{Compatible} if  for each $\pi \in \widehat{G}$, $\Phi(\pi)\in M_{d_{\pi}}(\mathbb{C})$. Here $M_{d_{\pi}}(\mathbb{C})$ denotes the set of all $d_{\pi}\times d_{\pi}$ complex matrices after identifying with $B(H_{\pi})$ for each ${\pi}$.
Recall that the set 
\begin{equation}
	\tilde{S}(G)= \{\gamma:\widehat{G} \rightarrow \cup_{\pi}M_{d_{\pi}}(\mathbb{C})\},
\end{equation}
is clearly a regular semi-group. Now  set 
\begin{equation}
	\Delta(G)	=\{\gamma:\widehat{G} \rightarrow \cup_{\pi}M_{d_{\pi}}(\mathbb{C}),\gamma \quad  compatible\}.
\end{equation}
Then we have the following;	
\begin{itemize}	
	\item[{[1]}]
	The problem under investigation is the algebraic regularity of the following topological semi-groups and of finding the maximal regular subsemigroup of  $ \widehat{P(G)}$.
	Since the non-commutative Fourier transform is known to be an isomorphism, it is clear that $\widehat{P(G)}$ is not algebraically regular. Moreover, the following problems are also relevant in this context,
	
	\item[{[2]}]The regularity of the associated semigroups $\tilde{S}(G) \&\Delta(G)$,
	
	\item[{[3]}] 
	Observe that these semi-groups are related as follows;
	\[
	\widehat{P(G)}	\subset \Delta(G)  \subset \tilde{S}(G).
	\]
\end{itemize}
\begin{theorem}	Let $G$ be a compact topological group Then $\tilde{S}(G)$ and $\Delta(G)$
	are regular semi-groups.
\end{theorem}
\begin{proof}
	It is well known that $\tilde(S)(G)$ and $\Delta(G)$ are semi-groups. In either case 
	regularity is easy to establish, as shown below.
	Let $\gamma \in \tilde{S}(G)$  (or $\Delta(G)$). For each $\pi \in \hat{G}$ let
	$\gamma^{\dagger}(\pi)$ be  the Moore-Penrose inverse of $\gamma(\pi).$ Clearly $\gamma^{\dagger}(\pi) \in \mathcal{M}(\hat{G})$. If $\gamma \in \Delta(G)$ so is $\gamma^{\dagger}$.
\end{proof}
\section{Minimal Regular Semigroups Containing $P(G)$}\label{sec5}
Next we consider the problem whether there are regular semigroups $\widetilde {\widehat{\Delta(G)}}$ such that
\begin{equation}
	\widehat{P(G)}	\subset \widetilde{\widehat{\Delta(G)}} \subset \Delta(G).
\end{equation}	
We restrict our attention to \textbf {compact Lie groups $G$} where new techniques  such as \textbf{Log-Ng positivity} \cite{DAM}  are available which is defined as follows:
\begin{definition}
	A compatible function $\gamma:\hat{G}\rightarrow M$ is called \textbf{Lo-Ng positive } if 
	\begin{equation}
		\Sigma_{\pi \in \Omega}	 d_{\pi}tr(\pi(g)\gamma (\pi) B(\pi)) \geq 0,
	\end{equation}	
	whenever 
	\begin{equation}
		\Sigma_{\pi \in \Omega}	 d_{\pi}tr(\pi(g)B(\pi)) \geq 0,
	\end{equation}
	for all $g\in G$ and the summation is taken over alrbitrary finite subsets $\Omega$ of $\widehat{G}$.
\end{definition}
\begin{theorem}(Theorem 4.3.2, The Lo-Ng Criterion\cite{DAM})
	Let $P(G)$ denote the class of regular probability measures on a compact Lie group $G$ and 
	$\gamma:\widehat{G}\rightarrow \textbf{M}(G)$ be a comptible mapping.Then $\gamma=\hat{\mu}$ if and only if  $\gamma$ is Lo-Ng  positive namely 
	\begin{equation}
		h_{n}(g)=\Sigma_{\pi\in S_{n}}z^{(n)}_{\pi}d_{\pi}	tr(\pi(g)\gamma(\pi)) \geq 0,
	\end{equation}	
	for all $g \in G$,where $ \#(S_{m}),\#(S_{n}) <\infty$ if $m<n$ and $\pi_{0}\in S_{n}$ for all $n$.
\end{theorem}
\begin{remark}
	\begin{itemize}
		\item[{[1]}]The above theorem is a non-commutative analogue of the celebrated Bochkner's theorem:
		Let $G$ be a locally compact abelian group and $\hat{G}$ be the dual group of characters. Let $F:\hat{G}\rightarrow \mathbb{C}$. Then $F$ is the Fourier transform of a measure $\mu,$
		
		\begin{equation}\nonumber
			\hat{\mu}(\chi)=\int_{G}\bar{\chi(g)}\mu(dg) F(x_{i}-x_{j}) \geq 0 \textrm{ if and only if } F(\hat{e})=1, \,F \textrm{ is continuous at }\hat{e},
		\end{equation}	
		
		\begin{equation}\nonumber
			\hat{\mu}(\pi)\psi=\int_{G} \pi(g^{-1})\psi \mu(dg),\, \pi \in \hat{G}.
		\end{equation}	
	\end{itemize}
\end{remark}
Now we have the following theorem, which is an easy consequence of Theorem 7.2 above.
\begin{theorem}
	Let $P(G)$ denote the class of regular probability measures on a compact Lie group $G$.Then $\mu \in P(G)$ is regular if and only if $\hat{\mu}^{\dagger}$ is Lo-Ng positive where $\hat{\mu}^{\dagger}(g)$ is the Moore-Penrose  inverse of $\hat{\mu}(g)$ for each $g\in G$.
\end{theorem}
\section{Moore-Penrose of Probability Measures}
We restrict our focus on on $P(G)$, where $G$ is a compact Lie group. It is well known that the Moore-Penrose inverse of matrices satisfy the following conditions which are as follows.
\begin{definition}
	Let $A$ be a real or complex matrix of finite order $n$. The Moore-Penrose inverse $A^{\nmid}$ satisfy the following conditions:
	\begin{itemize}
		\item[{[1]}] : $A^{\nmid}A A^{\nmid}= A^{\nmid}$\\
		
		\item[{[2]}]: $A A^{\nmid} A= A$\\
		
		and
		\item[{[3]}]:
		$AA^{\nmid} \& A^{\nmid} A$ are orthogonal projections.
	\end{itemize}
\end{definition}
The condition (3) above implies the uniqueness of $A^{\dagger}$. Without condition (3)  there will be infinitely many generalised inverses. 

Wendel introduced the following involution $\star$ in $P(G)$:
Let $\mu \in P(G)$. Then the involution $\mu^{\star}$ of $\mu $ is defined as
$$
\mu^{\star}(E)= \mu(E^{-1}), 
$$
for every Borel set $E$ in $G$.
\begin{theorem}{\cite{WEN}}
	$\mu \in P(G)$ is an idempotent if and only if
	\begin{itemize}
		\item[{[1]}] : $\mu^{2}=\mu$

		\item[{[2]}]: $\mu^{\star}= \mu$
	\end{itemize}
\end{theorem}
\section{Concluding Remarks}
This short article provide a few  interesting problems for future investigation. These are  as follows :
\begin{itemize}
	\item[{[1]}]   If the group $\mathcal{G}$ under consideration has only trivial finite subgroups, then the only generalised invertible elements are Diract point mass measures $\{\delta_{x}: x \in \mathcal{G}\}$.\\
	\item\item[{[2]}]  Theorem 5.2  shows that to deal with non abelian groups , one need more sophisticated 
	results established for nonhermitian doubly stochastic matrices by Plemons and Cline \cite{PLC}. This problem requires further investigation.\\
		
			\item\item[{[3]}] Interpretation of the random walks associated with the probability measures $\mu$ and $\mu^{\dagger}$ is to be looked into.
\end{itemize}
	\nocite{*}
	\section{Declarations}
	
	(1) Ethical Approval: Not applicable.\\
	
	(2) Competing interests: The author declares that there is no financial interest. There is personal interest of academic nature only.\\
	
	(3) Author's contribution: Not applicable, being a single-author paper.\\
	
	(4) Funding: There is no funding for the work done in this article.\\
	
	(5) Availability of data and materials: Not applicable.\\
	
\bibliographystyle{amsplain}
	\bibliography{Reqularity-2023-3}	
	
\end{document}